\documentclass{elsarticle}

\usepackage{amscd}
\usepackage{amsmath}
\usepackage{amssymb}
\usepackage{amsthm}
\usepackage{graphicx}
\usepackage[all]{xy}
\usepackage{kotex}
\usepackage{color}
\usepackage{enumerate}
\theoremstyle{plain}

\usepackage{nameref}
\newtheorem{Thm}{Theorem}[section]
\newtheorem{Lemma}[Thm]{Lemma}

\newtheorem*{Cor*}{Corollary}

\newtheorem*{definition*}{Definition}
\newtheorem*{Thm*}{Theorem}
\theoremstyle{remark}
\newtheorem{Rmk}[Thm]{\bf{Remark}}

\usepackage{color}							

\journal{Computers \& Mathematics with Applications}
\begin{document}
\begin{frontmatter}
\title{
A New Condition for the Concavity Method of Blow-up Solutions to Semilinear Heat Equations
}

\author[syc]{Soon-Yeong Chung}
\address[syc]{Department of Mathematics and Program of Integrated Biotechnology, Sogang University, Seoul 04107, Korea}
\ead{sychung@sogang.ac.kr}

\author[mjc]{Min-Jun Choi\corref{cjhp}}
\address[mjc]{Department of Mathematics, Sogang University, Seoul 04107, Korea}
\cortext[cjhp]{Corresponding author}
\ead{dudrka2000@sogang.ac.kr}

\begin{abstract}

In this paper, we consider the semilinear heat equations under Dirichlet boundary condition
\begin{equation*}
	\begin{cases}
		u_{t}\left(x,t\right)=\Delta u\left(x,t\right)+f(u(x,t)), & \left(x,t\right)\in \Omega\times\left(0,+\infty\right),\\
		u\left(x,t\right)=0, & \left(x,t\right)\in\partial \Omega\times\left[0,+\infty\right),\\
		u\left(x,0\right)=u_{0}\geq0, & x\in\overline{\Omega},
	\end{cases}
\end{equation*}
where $\Omega$ is a bounded domain of $\mathbb{R}^{N}$ $(N\geq1)$ with smooth boundary $\partial\Omega$. The main contribution of our work is to introduce a new condition 
\begin{center}
	(C)$\hspace{1cm} \alpha \int_{0}^{u}f(s)ds \leq uf(u)+\beta u^{2}+\gamma,\,\,u>0$
\end{center}
for some $\alpha, \beta, \gamma>0$ with $0<\beta\leq\frac{\left(\alpha-2\right)\lambda_{0}}{2}$, where $\lambda_{0}$ is the first eigenvalue of Laplacian $\Delta$, and we use the concavity method to obtain the blow-up solutions to the semilinear heat equations. In fact, it will be seen that the condition (C) improves the conditions known so far.
\end{abstract}

\end{frontmatter}
\section{Introduction}

In this paper, we discuss the blow-up solutions for the following semilinear heat equations 
\begin{equation}\label{equation}
\begin{cases}
u_{t}\left(x,t\right)=\Delta u\left(x,t\right)+f\left(u\left(x,t\right)\right), & \left(x,t\right)\in \Omega\times\left(0,+\infty\right),\\
u\left(x,t\right)=0, & \left(x,t\right)\in\partial \Omega\times\left[0,+\infty\right),\\
u\left(x,0\right)=u_{0}\left(x\right)\geq0, & x\in\overline{\Omega},
\end{cases}
\end{equation}
where $\Omega$ is a bounded doamin in $\mathbb{R}^{N}(N\geq1)$ with smooth boundary $\partial \Omega$ and $f$ is locally Lipschitz continuous on $\mathbb{R}$, $f(0)=0$ and $f(u)>0$ for $u>0$. Moreover, the initial data $u_{0}$ is a non-negative function in $C^{1}(\overline{\Omega})$ satisfying that $u_{0}(x)=0$ on $\partial\Omega$.

The blow-up solutions to \eqref{equation} have been studied by many authors. To determine of sufficient conditions for the blow-up of solutions to \eqref{equation}, they assumed some conditions on the nonhomogeneous term $f$, the $C^{1}$ condition in \cite{FM}, $\int_{a}^{\infty}\frac{ds}{f(s)}<\infty$ in \cite{CF, GV}, some homogeneity in \cite{KST}, the nonnegativity of $f'$ with convexity in \cite{M}, respectively. Commonly, they used a differential inequality technique and a comparison principle to obtain blow-up solutions. On the other hand, Levine and Payne \cite{L, LP, LPa} introduced a new and elegant tool for deriving estimates and giving criteria for blow-up, which called concavity method. Afterwords, using concavity method, many authors derived blow-up solution to \eqref{equation} and whenever they did so, they introduced some conditions for $f$. For example, in \cite{PP}, the authors gave the condition for $f$ as follows: for some $\epsilon>0$,
\begin{equation}\label{PP}
(2+\epsilon)\int_{0}^{u}f(s)ds\leq uf(u),\,\,u>0.
\end{equation}
In \cite{BB}, the condition \eqref{PP} was relaxed to
\begin{equation}\label{IBB}
(2+\epsilon)\int_{0}^{u}f(s)ds\leq uf(u) + c^2,\,\,u>0.
\end{equation}

Looking into the above conditions \eqref{PP} and \eqref{IBB} more closely, we can see that there are independent of the eigenvalue which depends on the domain $\Omega$. However, our main proof consists of a series of inequalities and the Poincare inequality including the eigenvalue. From this fact, we can expect to develop an improved condition which refines \eqref{PP} or \eqref{IBB}, and depends on the domain $\Omega$. Being motivated by this point of view, we develop a new condition as follows: for some $\alpha, \beta, \gamma>0$,
\[
\hbox{(C)\hspace{1cm}  $\alpha \int_{0}^{u}f\left(s\right)ds\leq uf(u) + \beta u^{2} + \gamma,\,\,u>0$},
\]
where $0<\beta\leq\frac{\left(\alpha-2\right)\lambda_{0}}{2}$ and $\lambda_{0}$ is the first eigenvalue of the Laplacian $\Delta$. 

The main theorem of this paper is as follows:
\begin{Thm*}\label{BlowB}
	Let the function $f$ satisfy the condition (C). If the initial data $u_{0}$ satisfies
	\begin{equation}\label{11}
	-\frac{1}{2}\int_{\Omega}\left|\nabla u_{0}\left(x\right)\right|^{2}dx+\int_{\Omega}\left[\int_{0}^{u_{0}\left(x\right)}f(s)ds-\gamma\right]dx>0,
	\end{equation}
	then the nonnegative classical solution $u$ to the equation \eqref{equation} blows up at finite time $T^{*}$ in the sense of
	\[
	\lim_{t\rightarrow T^{*}}\int_{0}^{t}\sum_{x\in\overline{S}}u^{2}\left(x,s\right)ds=+\infty,
	\]
	where $\gamma$ is the constant in the condition (C).
\end{Thm*}

Next, in Section \ref{BCM}, we discuss the blow-up solutions using concavity method with the condition (C).

\section{Blow-Up: Concavity Method}\label{BCM}

In this section, we discuss the blow-up phenomena of the solution to \eqref{equation} by using concavity method, which is the main part of this paper. In fact, it is well known that the equation \eqref{equation} has a classcal solution. (See \cite{B} for more details).

The following lemma is used to prove the main thoerem.

\begin{Lemma}[See \cite{E}]\label{eigenvalue}
	There exist $\lambda_{0}>0$ and $\phi_{0}(x)\in C^{2}(\Omega)$ such that 
	\[
	\begin{cases}
	-\Delta\phi_{0}\left(x\right)=\lambda_{0}\phi_{0}\left(x\right), & x\in \Omega,\\
	\phi_{0}\left(x\right)=0, & x\in\partial \Omega,\\
	\end{cases}
	\]
Moreover, $\lambda_{0}$ is given by
\[
\lambda_{0} = \min_{u\in\mathcal{A},u\not\equiv0}\frac{\int_{\Omega}\left|\nabla u\left(x\right)\right|^{2}dx}{\int_{\Omega}\left|u\left(x\right)\right|^{2}dx}\\
\]
where $\mathcal{A}:=\left\{ u\,:\, u\in C^{2}(\Omega), u\not\equiv0, u=0 \mbox{ for } x\in \partial \Omega\right\}$.
\end{Lemma}
In the above, the number $\lambda_{0}$ is the first eigenvalue of $\Delta$ and $\phi_{0}$ is a corresponding eigenfunction.

Now, we state and prove our main result.

\begin{Thm*}\label{BlowB}
Let the function $f$ satisfy the condition (C). If the initial data $u_{0}$ satisfies
\begin{equation}\label{11}
-\frac{1}{2}\int_{\Omega}\left|\nabla u_{0}\left(x\right)\right|^{2}dx+\int_{\Omega}\left[\int_{0}^{u_{0}\left(x\right)}f(s)ds-\gamma\right]dx>0,
\end{equation}
then the nonnegative classical solution $u$ to the equation \eqref{equation} blows up at finite time $T^{*}$ in the sense of
\[
\lim_{t\rightarrow T^{*}}\int_{0}^{t}\sum_{x\in\overline{S}}u^{2}\left(x,s\right)ds=+\infty,
\]
where $\gamma$ is the constant in the condition (C).
\end{Thm*}
\begin{proof}
Now, we define a functional $J$ by
\[
J\left(t\right):=-\frac{1}{2}\int_{\Omega}\left|\nabla u\left(x,t\right)\right|^{2}dx+\int_{\Omega}\left[F\left(u\left(x,t\right)\right)-\gamma\right]dx,\;\;\;t\geq0,
\]
where $F(u):=\int_{0}^{u}f(s)ds$.

Then by \eqref{11},
\[
J\left(0\right)=-\frac{1}{2}\int_{\Omega}\left|\nabla u_{0}\left(x\right)\right|^{2}dx+\int_{\Omega}\left[F(u_{0}\left(x\right))-\gamma\right]dx>0.
\]
and we can see that 
\begin{equation}\label{Jtt}
J(t)=J(0)+\int_{0}^{t}\frac{d}{dt}J(s)ds=J(0)+\int_{0}^{t}\int_{\Omega}u_{t}^{2}(x,s)dxds.
\end{equation}

Now, we introduce a new function
\begin{equation}\label{It}
I\left(t\right)=\int_{0}^{t}\int_{\Omega}u^{2}\left(x,s\right)dxds+M,\,t\geq 0,
\end{equation}
where $M>0$ is a constant to be determined later. Then it is easy to see that 
\begin{equation}\label{15}
\begin{aligned}
I'\left(t\right) & =  \int_{\Omega}u^{2}\left(x,t\right)dx\\
& =  \int_{\Omega}\int_{0}^{t}2u\left(x,s\right)u_{t}\left(x,s\right)dsdx+\int_{\Omega}u_{0}^{2}\left(x\right)dx.
\end{aligned}
\end{equation}

Then we use integration by parts, the condition (C), Lemma \ref{eigenvalue}, and \eqref{Jtt} in turn to obtain
\begin{equation}\label{14}
\begin{aligned}
I''\left(t\right)&= \frac{d}{dt}\int_{\Omega}u^{2}\left(x,t\right)dx\\
&=\int_{\Omega}2u\left(x,t\right)u_{t}\left(x,t\right)dx\\
&=\int_{\Omega}2u\left(x,t\right)\left[\Delta u\left(x,t\right) + f\left(u(x,t)\right) \right]dx\\
&=-2\int_{\Omega}\left|\nabla u\left(x,t\right)\right|^{2}dx + \int_{\Omega}2u(x,t)f\left(u(x,t)\right)dx\\
&\geq-2\int_{\Omega}\left|\nabla u\left(x,t\right)\right|^{2}dx + \int_{\Omega}2\left[\alpha F(u(x,t))-\beta u^{2}(x,t)-\alpha \gamma\right] dx\\
&=2\alpha\left[-\frac{1}{2}\int_{\Omega}\left|\nabla u(x,t)\right|^{2}dx+\int_{\Omega}[F(u(x,t))-\gamma] dx\right]\\
&+(\alpha-2)\int_{\Omega}\left|\nabla u\left(x,t\right)\right|^{2}dx - 2\beta \int_{\Omega}u^{2}(x,t)dx\\
&\geq 2\alpha J(t)+\left[(\alpha-2)\lambda_{0}-2\beta\right]\int_{\Omega}u^{2}(x,t)dx\\
&\geq 2\alpha \left[J(0)+\int_{0}^{t}\int_{\Omega}u_{t}^{2}(x,s)dxds\right].
\end{aligned}
\end{equation}

Using the Schwarz inequality, we obtain
\begin{equation}\label{16}
\begin{aligned}
& I'\left(t\right) ^{2}\\
& \leq 4\left(1+\delta\right)\left[\int_{\Omega}\int_{0}^{t}u\left(x,s\right)u_{t}\left(x,s\right)dsdx\right]^{2}+\left(1+\frac{1}{\delta}\right)\left[\int_{\Omega}u_{0}^{2}\left(x\right)dx\right]^{2}\\
& \leq  4\left(1+\delta\right)\left[\int_{\Omega}\left(\int_{0}^{t}u^{2}\left(x,s\right)ds\right)^{\frac{1}{2}}\left(\int_{0}^{t}u_{t}^{2}\left(x,s\right)ds\right)^{\frac{1}{2}}dx\right]^{2}\\
&+\left(1+\frac{1}{\delta}\right)\left[\int_{\Omega}u_{0}^{2}\left(x\right)dx\right]^{2}\\
& \leq  4\left(1+\delta\right)\left(\int_{\Omega}\int_{0}^{t}u^{2}\left(x,s\right)dsdx\right)\left(\int_{\Omega}\int_{0}^{t}u_{t}^{2}\left(x,s\right)dsdx\right)\\
&+\left(1+\frac{1}{\delta}\right)\left[\int_{\Omega}u_{0}^{2}\left(x\right)dx\right]^{2},
\end{aligned}
\end{equation}
where $\delta>0$ is arbitrary. Combining the above estimates \eqref{It}, \eqref{14}, and \eqref{16}, we obtain that for $\xi=\delta=\sqrt{\frac{\alpha}{2}}-1>0$,
\begin{equation*}
\begin{aligned}
&I''\left(t\right)I\left(t\right)-\left(1+\xi\right)I'\left(t\right)^{2}\\
& \geq 2\alpha\left[J\left(0\right)+\int_{0}^{t}\int_{\Omega}u_{t}^{2}\left(x,s\right)dxds\right]\left[\int_{0}^{t}\int_{\Omega}u^{2}\left(x,s\right)dxds+M\right]\\
& -  4\left(1+\xi\right)\left(1+\delta\right)\left[\int_{\Omega}\int_{0}^{t}u^{2}\left(x,s\right)dsdx\right]\left[\int_{\Omega}\int_{0}^{t}u_{t}^{2}\left(x,s\right)dsdx\right]\\
& -  \left(1+\xi\right)\left(1+\frac{1}{\delta}\right)\left[\int_{\Omega}u_{0}^{2}\left(x\right)dx\right]^{2}\\
& \geq  2\alpha M\cdot J\left(0\right)-\left(1+\xi\right)\left(1+\frac{1}{\delta}\right)\left[\int_{\Omega}u_{0}^{2}\left(x\right)dx\right]^{2}.
\end{aligned}
\end{equation*}

Since $J\left(0\right)>0$ by assumption, we can choose $M>0$ to be large enough so that
\begin{equation}\label{17}
I''(t)I\left(t\right)-\left(1+\xi\right)I'\left(t\right)^{2}>0.
\end{equation}

This inequality \eqref{17} implies that for $t\geq0$,
\[
\frac{d}{dt}\left[\frac{I'\left(t\right)}{I^{\xi+1}\left(t\right)}\right]>0  \hbox{  i.e. }I'\left(t\right)\geq\left[\frac{I'\left(0\right)}{I^{\xi+1}\left(0\right)}\right]I^{\xi+1}\left(t\right).
\]

Therefore, it follows that $I\left(t\right)$ cannot remain finite for all $t>0$. In other words, the solutions $u\left(x,t\right)$  blow up in finite time $T^{*}$.
\end{proof}


\begin{Rmk}
The above blow-up time can be estimated roughly. Taking
\[
M:=\frac{\frac{\alpha}{\alpha-2}\left(1+\sqrt{\frac{\alpha}{2}}\right)\left[{\int_{\Omega}u_{0}^{2}\left(x\right)dx}\right]^{2}}{2\alpha\left[-\frac{1}{2}\int_{\Omega}\left|\nabla u_{0}\left(x\right)\right|^{2}dx+\int_{\Omega}\left[F(u_{0}\left(x\right))-\gamma\right]dx\right]},
\]
we see that
\[
\begin{cases}
I'\left(t\right)\geq\left[\frac{\int_{\Omega}u_{0}^{2}\left(x\right)dx}{M^{\xi+1}}\right]I^{\xi+1}\left(t\right), & t>0,\\
I\left(0\right)=M,
\end{cases}
\]
which implies
\[
I\left(t\right)\geq\left[\frac{1}{M^{\xi}}-\frac{\xi\int_{\Omega}u_{0}^{2}\left(x\right)dx}{M^{\xi+1}}t\right]^{-\frac{1}{\xi}}
\]
where $\xi=\sqrt{\frac{\alpha}{2}}-1>0$. Then the blow-up time $T^{*}$ satisfies
\begin{equation*}\label{22}
0<T^{*}\leq\frac{M}{\xi\int_{\Omega}u_{0}^{2}\left(x\right)dx}.
\end{equation*}
\end{Rmk}
\section*{Conflict of Interests}
\noindent The authors declare that there is no conflict of interests regarding the publication of this paper.

\section*{Acknowledgments}
\noindent This work was supported by Basic Science Research Program through the National Research Foundation of Korea(NRF) funded by the Ministry of Education (NRF-2015R1D1A1A01059561). 

\end{document}